\documentclass{amsart}
\usepackage{amscd, stackrel}
\usepackage{amssymb} \usepackage[all, knot]{xy} \usepackage{bbm}
\usepackage{amssymb,latexsym, mathtools,tikz}
\usepackage{enumerate}
\usepackage{graphicx}
\xyoption{all}
\xyoption{arc}
\usepackage{hyperref}
\usepackage[mathscr]{euscript}
\usepackage{extpfeil}

\def\height{\operatorname{height}}

\def\len{\operatorname{\ell}}

\def\del{\partial}

\def\cube#1#2#3#4#5#6#7#8{
& #5 \ar[rr] \ar[dl] \ar@{-}[d] && #6 \ar[dd] \ar[dl] \\
#1 \ar[rr] \ar[dd]  & \ar[d] & #2 \ar[dd] \\
& #7 \ar@{-}[r] \ar[dl] & \ar[r] & #8 \ar[dl] \\
#3 \ar[rr] && #4 \\
}

\def\smsh{\wedge}

\def\chr{\operatorname{char}}

\def\image{\operatorname{im}}
\def\im{\image}
\def\ker{\operatorname{ker}}

\def\cP{\mathcal P}

\def\cC{\mathcal C}

\def\coker{\operatorname{coker}}

\def\dm{\operatorname{dim}}
\def\codim{\operatorname{codim}}
\def\rank{\operatorname{rank}}

\def\fl{{\mathsf{fl}}}
\def\Tor{\operatorname{Tor}}

\def\Spec{\operatorname{Spec}}

\def\into{\hookrightarrow}
\def\onto{\twoheadrightarrow}
\def\ann{\operatorname{ann}}

\def\b{\beta}

\newcommand{\G}{\mathbb{G}}

\newcommand{\Z}{\mathbb{Z}}

\newcommand{\fp}{{\mathfrak p}}
\newcommand{\fm}{{\mathfrak m}}

\numberwithin{equation}{section}

\theoremstyle{plain} 
\newtheorem{thm}[equation]{Theorem}
\newtheorem{thm-conj}[equation]{Theorem-Conjecture}
\newtheorem{defn-conj}[equation]{Definition-Conjecture}

\newtheorem*{introthm*}{Theorem}

\newtheorem{cor}[equation]{Corollary}
\newtheorem{lem}[equation]{Lemma}

\newtheorem{prop}[equation]{Proposition}

\newtheorem{conj}[equation]{Conjecture}

\theoremstyle{definition}
\newtheorem{defn}[equation]{Definition}

\theoremstyle{remark}
\newtheorem{rem}[equation]{Remark}

\newtheorem{notation}[equation]{Notation}

\def\Perf{\operatorname{Perf}}

\newcommand{\Hom}{\operatorname{Hom}}

\newcommand{\supp}{\operatorname{supp}}

\newcommand{\xra}[1]{\xrightarrow{#1}}

\newcommand{\pd}{\operatorname{pd}}

\renewcommand{\cC}{\mathcal{C}}
\newcommand{\sC}{\operatorname{s} \mathcal{C}}

\def\L{\Lambda}

\def\wt#1{\widetilde{#1}}

\def\and{ \text{ and } }

\def\can{\mathrm{can}}

\def\op{{\mathrm{op}}}

\def\G{\Gamma}

\renewcommand{\op}{{\text{op}}}

\def\Chzero{\operatorname{Ch}_{\geq 0}}

\title{The Total Rank Conjecture in Characteristic Two}
\author{Keller VandeBogert}
\author{Mark E.~Walker}

\begin{document}

\maketitle

\begin{abstract}
    The Total Rank Conjecture is a coarser version of the Buchsbaum-Eisenbud-Horrocks conjecture which, loosely stated, predicts that modules with large annihilators must also have large syzygies. This conjecture was proved by the second author for rings of odd characteristic by taking advantage of the Adams operations on the category of perfect complexes with finite length homology. In this paper, we prove a stronger form of the Total Rank Conjecture for rings of characteristic two. This result may be seen as evidence that the Generalized Total Rank Conjecture (which is known to be false in odd characteristic) actually holds in characteristic two. We also formulate a meaningful extension of the Total Rank Conjecture over non Cohen-Macaulay rings and prove that this conjecture holds for any algebra over a field (independent of characteristic).
\end{abstract}

\section{Introduction}
The ``Rank Conjectures'' are a collection of related conjectures in algebra and topology.
On the algebraic side, they refer to conjectural lower bounds on the ranks of the free modules occurring in certain types of chain complexes.
On the topological side, they refer to conjectural lower bounds on the ranks of the cohomology of spaces that admit free group actions.
For an instance of the latter, Carlsson conjectures \cite{carlsson1986free} that if an elementary abelian $2$-group of rank $n$ acts freely and cellularly
on a CW complex $X$, then the sum of the ranks of the $\Z/2$-homology of $X$ must be at least $2^n$ --- in other words, the larger the group that acts (freely), the larger
the homology must be. 

This paper focuses mostly on the algebraic rank conjectures, although we do discuss Carlsson's conjecture in Section \ref{sec:Carlsson}  below. Informally, such rank conjectures may be interpreted as saying that modules with large codimension must have more relations than modules with smaller codimension. 
In detail, over a local ring $(R , \fm , k)$ a minimal presentation of an $R$-module $M$ is given by a right exact sequence
$R^{\beta_1} \to R^{\beta_0} \to M \to 0$, with $\beta_0$ equal to the number of generators and $\beta_1$ equal to the number of relations.
It can be shown that $\beta_1 \geq \codim_R(M) := \height(\ann_R(M))$.
In other words, the minimal number of generators of the first syzygy of $M$ is at least its codimension; indeed, this bound is sharp since the quotient $M = R/(a_1 , \dots , a_d)$ by a regular sequence $a_1 , \dots ,a_d$ achieves equality for this lower bound. The Buchsbaum-Eisenbud-Horrocks (BEH)
Conjecture posits that this case in fact has the smallest possible number of relations for a fixed codimension:

\begin{conj}[BEH Conjecture]
    Let $R$ be a commutative Noetherian ring with connected spectrum and $M$ any nonzero $R$-module of finite projective dimension. Then for any projective resolution
    $$
    0 \to P_d \to \cdots \to P_1 \to P_0 \to M \to 0
    $$
  of $M$,  there is an inequality
    $$
    \rank_R (P_i) \geq \binom{\codim_R(M)}{i}.
    $$
\end{conj}

Upon localizing at a minimal prime $\fp$ in the support of $P$, the BEH conjecture reduces to the following special case of it:

\begin{conj}[Local BEH Conjecture]  If $(R, \fm, k)$ is local and $M$ is a non-zero $R$-module of finite length and finite projective
  dimension, then
  $$
  \b_i(M) \geq {\dm(R) \choose i}
  $$
  where $\b_i(M)$, the {\em $i$-th Betti number of $M$}, is the rank of the $i$-th free module in its minimal free resolution. 
\end{conj}

The BEH conjecture remains wide open despite many years of study; see for instance \cite{burman2011chang}, \cite{chang1997betti}, \cite{chang2000betti},
\cite{charalambous1991betti}, \cite{charalambous1990betti}, \cite{dugger2000betti}, \cite{erman2010special}, \cite{hochster2005lower}, \cite{santoni1990horrocks}.
A closely-related but weaker form of the BEH conjecture, first proposed by Avramov and stated formally in \cite{avramov1993lower},
concerns the sum of the Betti numbers:

\begin{conj}[Avramov's Total Rank Conjecture]
  Under the same assumptions as in the BEH conjecture, we have
  $$
\sum_{i =0}^d \rank_R (P_i) \geq 2^{\codim_R(M)}.
$$
\end{conj}

As with the BEH conjecture, the Total Rank Conjecture is equivalent to one of its special cases:

\begin{conj}[Local Total Rank Conjecture]
  If $(R, \fm, k)$ is local and $M$ is a non-zero, finite length $R$-module, then
  $$
  \beta^R(M) \geq 2^{\dm(R)}
  $$
  where $\beta^R(M) = \sum_i \beta_i^R(M)$ is the total Betti number of $M$. 
\end{conj}

The Total Rank Conjecture was proven by the second author \cite{walker2017total} in the following cases:
\begin{enumerate}
    \item when $R$ is locally a complete intersection and $M$ is $2$-torsion free, or
    \item when $R$ has characteristic $p$ for an \emph{odd} prime $p$. 
    \end{enumerate}

    It is a consequence of the New Intersection Theorem that if a local ring $R$ admits a module $M$ of finite length and projective dimension, then $R$ must be Cohen-Macaulay;
    that is,  the Local Total Rank Conjecture is vacuous for non-Cohen-Macaulay local rings.
    By enlarging the types of complexes considered, one may formulate a meaningful version of this conjecture even in the non-Cohen-Macaulay case:

\begin{defn} For a Noetherian commutative ring $R$, a {\em short  complex} is a non-exact complex of finite rank projective
  $R$-modules of the form
  $$
  P = \left(\cdots \to 0 \to P_c \to \cdots \to P_1 \to P_0 \to 0 \to
    \cdots\right)
  $$
  and such that $c = \codim_R(H(P)) := \max_j\{ \codim_R H_j(P)\}$.

    For a local ring $(R, \fm, k)$, a {\em tiny complex} is a short complex $P$ such that $\codim_R(H(P)) = \dm(R)$; that is, a tiny complex is a non-exact complex of finite
  rank free $R$-modules  of  form
  $$
  P = \left(\cdots \to 0 \to P_{\dm(R)} \to \cdots \to P_1 \to P_0 \to 0 \to
    \cdots\right)
  $$
  having finite length homology.
\end{defn}

Short complexes are so-named since in general, for a non-exact complex $P$ of projective modules
concentrated in homological degrees $[0, c]$, the New Intersection Theorem implies that $c \geq \codim_R(H(P))$.
Tiny complexes are both short and of smallest possible dimension. 

\begin{conj}[The Total Rank Conjecture for Short Complexes] For any short complex $P$ over $R$, we have
  $\sum_i \rank_R(P_i) \geq  2^{\codim_R(H(P))}$.  In particular, for any tiny complex $P$ over a local ring $R$, we have $\sum_i \rank_R(P_i) \geq  2^{\dm R}$. 
\end{conj}

In this non Cohen-Macaulay setting, the Total Rank Conjecture for Short Complexes predicts that the short complexes of minimal rank arise as Koszul complexes on a system of parameters. As with the previous conjectures, the Total Rank Conjecture for Short Complexes reduces to the special case of tiny complexes,
as is seen by localizing at a minimal prime in the support of a short complex.
It generalizes the original Total Rank Conjecture: Given a non-zero module of finite length and finite projective dimension over a local ring, its minimal free resolution is tiny.

One of the main results of this paper is a proof of the  Total Rank Conjecture in characteristic $2$. In fact, we prove slightly more:

\begin{thm}\label{thm:char2TotalRankC} Suppose $(R, \fm, k)$ is a local ring with $\chr(R) = 2$ and $P$ is a complex of finite rank free $R$-modules of the form
  $$
  P = \left(\cdots \to 0 \to P_{\dm(R)+1} \to P_{\dm(R)} \to \cdots \to P_1 \to P_0 \to 0 \right)
  $$
  having finite length homology. Then $\sum_j \rank_R(P_j) \geq 2^{\dm(R)}$. In particular, the Total Rank Conjecture (both the original version and the one for short complexes) holds for rings of characteristic two.  
\end{thm}

Before describing the other results of this paper, we briefly describe 
the main new idea that lies at the heart of the proof of this theorem.

In the second author's paper\cite{walker2017total}, the second tensor, symmetric, and exterior powers of a complex 
of free modules $F$, written as $T^2_R(F) = F \otimes_R F$, $S^2_R(F)$ and $\L^2_R(F)$, respectively,  
play vital roles. When $2$ is invertible in $R$, one may define $S^2_R(F)$ and $\L^2_R(F)$
to be the eigenspaces of the evident involution operator on $T^2_R(F)$, and the invertibility of $2$ implies that there is a natural splitting 
$$T^2_R (F) = S^2_R (F) \oplus \L^2_R (F).$$
Combining this splitting with the fact that 
$\chi S^2_R(F) - \chi \L^2_R(F)$ is equal to the Euler characteristic of the second Adams operation leads quickly to a proof of the Total Rank Conjecture in the cases listed above.  

When $\chr(R) = 2$, two related problems arise. First, the ``na\"ive" definitions of $S^2$ and $\Lambda^2$ for complexes as eigenspaces for the involution operators are badly behaved --- for instance, they don't preserve quasi-isomorphisms. This problem, however, can be circumvented by employing the Dold-Kan correspondence to pass to \emph{simplicial} modules, for which well-behaved symmetric and exterior power functors always exist. More precisely, given a nonnegatively graded complex of free $R$-modules $F$, we may convert $F$ into a simplicial $R$-module via Dold-Kan, apply either $S^2$ or $\L^2$ degree-wise, then pass back to the category of complexes. This process yields canonically defined functors $\wt{S^2}$ and $\wt{\L^2} $ that are now guaranteed to preserve quasi-isomorphisms on the subcategory of bounded complexes of finite rank free $R$-modules.

The second, more subtle problem is that the canonical short exact sequence
$$
0 \to \L^2_R P \to P \otimes_R P \to S^2_R P \to 0,
$$
defined on the category of finite rank free $R$-modules $P$, admits a natural splitting only when $2$ is invertible in $R$.
In other words, when $\chr(R) = 2$, even though we may still associate to a complex $F$ a short exact sequence
    $$
    0 \to \wt{\L^2} F \to \wt{T^2} F \to \wt{S^2} F \to 0
    $$
    of complexes, 
the boundary map in the associated long exact sequence in homology may fail to be $0$. 
This invalidates the identity $h ( \wt{T^2}(F)) = h (\wt{S^2}(F))  + h(\wt{\L^2} F )$ (where $h$ denotes the length of the total homology of a complex) and the original 
argument is doomed.\footnote{ It remains true that $ h(\wt{T^2}(F))  \leq h (\wt{S^2}(F) )  +  h( \wt{\L^2} F )$, but the inequality is in the
wrong direction for the original proof to work.}

With this in mind, we may now describe the key new idea:
When $\chr(R) = 2$, although we lose such a natural splitting, 
we instead gain access to a natural short exact sequence unique to the characteristic $2$ case:
$$
0 \to FP \to S^2_R P \to \L^2_R P \to 0,
$$
where $F$ denotes extension of scalars along the Frobenius map; in other words, characteristic $2$ is distinguished for being the only case where there is an honest categorification of the Adams operation identity $\psi^2 ([F]) = [S^2 F] - [\L^2 F]$.
This sequence is an essential ingredient in many of the proofs in this paper.

One surprising aspect of Theorem \ref{thm:char2TotalRankC} is that it excludes in characteristic $2$ the kinds of counterexamples to the ``Generalized Total Rank Conjecture'' found
by Iyengar and the second author. The Generalized Total Rank Conjecture predicted that for \emph{any} bounded complex of free modules
\begin{equation} \label{E414}
  P = (\cdots \to 0 \to P_m \to \cdots \to P_0 \to 0)
  \end{equation}
  over a local ring $R$ having  finite length homology,
one must have $\sum_j \rank_R(P_j) \geq 2^{\dm(R)}$. In \cite{iyengar2018examples} it is shown that, so long as $\dm(R) \geq 8$ and $\chr(k) \ne 2$, there exists a complex (even over a regular local ring) with $m = \dm(R) + 1$, $H_0(P) = H_1(P) = k$ and $H_j(P) = 0$ for all $j \notin \{0,1\}$ such that $\sum_j \rank_j P_j < 2^{\dm(R)}$. 
Theorem \ref{thm:char2TotalRankC} shows that no such example exists in characteristic $2$.
Moreover, even if $m = \dm(R) + 2$, there is some evidence that the Generalized Total Rank Conjecture may hold in characteristic $2$; see Theorem \ref{thm:012homology} below.

By combing Theorem \ref{thm:char2TotalRankC}, previous results of the second author, and a reduction to characteristic $p$ argument (``Hochster's metatheorem"),
we prove:

\begin{thm}\label{thm:shortComplexAnalog}
  The Total Rank Conjecture (both the original formulation and that for short complexes) holds for all noetherian $k$-algebras $R$. 
  \end{thm}

  The Total Rank Conjecture does remain open for rings of mixed characteristic. Partial results in the mixed characteristic setting were previously obtained by
  the second author for Cohen-Macaulay rings that are ``quasi-Roberts rings'' (e.g.,
  compete intersections) and that have residual characteristic not $2$. Another consequence of our main Theorem \ref{thm:char2TotalRankC} is a slightly weaker result, one that is
  off by a factor of $\frac12$ from the conjectural lower bound, that holds for all mixed characteristic rings of residual
  characteristic equal to $2$:

  \begin{cor} If $(R, \fm, k)$ is  local ring with $\chr(k) = 2$ and $P$ is a tiny complex over $R$, then $\sum_j \rank_R(P_j) \geq 2^{\dm(R)-1}$. In particular, for any module $M$
    of finite length over such a ring we have $\beta(M) \geq 2^{\dm(R)-1}$. 
  \end{cor}

In this paper we also address the question of when the lower bound in the Local Total Rank Conjecture is met. 
  For any local ring $(R, \fm, k)$ of dimension $d$, the Koszul complex $K$ on any system of parameters forms a tiny complex whose total rank is exactly $2^d$. This motivates:

  \begin{conj}[Extended Total Rank Conjecture] \label{ExtendedConjecture}
    If $P$ is a tiny complex over a local ring $R$ of dimension $d$, then $\sum_i \rank(P_i) \geq 2^d$ with equality holding if and only if $P$ is isomorphic
    to a Koszul complex on a system of parameters. In particular, for
    a local ring $(R, \fm)$ and a module $M$ of finite
    length, if $\beta(M) = 2^{\dm R}$, then $M \cong R/(x_1, \dots, x_d)$ for some system of parameters $x_1, \dots, x_d \in \fm$. 
  \end{conj}

  The second author proved this extended conjecture for local Cohen-Macaulay, quasi-Roberts rings of odd characteristic \cite{walker2017total}.
  In this paper, we extend this result to the even characteristic case:

\begin{thm} \label{ThmRobertsRings} Suppose $R$ is a local ring that is Cohen-Macaulay of characteristic $2$ and that $R$ is  quasi-Roberts (for example, a complete intersection).
  If $M$ is an $R$-module of finite length, then $\beta(M) \geq 2^{\dm(R)}$ and equality holds if and only if $M$ is isomorphic to the quotient of $R$ by a system of parameters.
\end{thm}

\subsection{Organization of this paper}  
In Section \ref{sec:background}, we give some necessary background. This includes a discussion of one our main tools: the \emph{Dold-Kan
  correspondence}, which gives a canonical method of extending endo-functors of $R$-modules to endo-functors of $R$-complexes. We also recall the notion of Dutta multiplicities and record some results
about the behavior of Dutta multiplicities on tiny complexes. 

In Section \ref{sec:TotalRankRobertsChar2} we prove Theorem \ref{ThmRobertsRings} and
in Section \ref{sec:genTRC}, we prove Theorem \ref{thm:char2TotalRankC}.

In Section \ref{sec:TRCOddChar} we prove that the Total Rank Conjecture for Short Complexes holds for algebras over a field (Theorem  \ref{thm:shortComplexAnalog}). 
The proof for fields of odd characteristic follows the proof of the original Total Rank Conjecture  given in \cite{walker2017total},  with minor modifications.  In the characteristic $0$ case, we employ a version of Hochster's metatheorem \cite{hochsterSurvey} proved by Kurano-Roberts \cite{KRAdamsOps}.

Finally, in Section \ref{sec:Carlsson} we apply our techniques to deduce some special cases of the conjecture of Carlsson mentioned above.

\subsection{Notation} \label{sec:notation}

\begin{notation} 
  For any commutative Noetherian ring $R$, we write $\Perf(R)$ for the category of all bounded complexes of finitely generated and projective $R$-modules.
For an $R$-module $M$, we write $\len(M) = \len_R(M)$ for its length (the length of any finite composition series of $M$ or $\infty$  if no such
series exists). 
  The subcategory of $\Perf(R)$ consisting of those complexes $P$ such that $\len_R H_i(P) < \infty$ for all $i$ 
is written as $\Perf^\fl(R)$. For $P \in \Perf^\fl(R)$, we set 
    $$
    h_i (P) := \len_R H_i (P), 
    \quad
    h (P) := \sum_i h_i (P), \quad 
    \text{ and } \quad \chi(P) := \sum_i (-1)^i h_i(P).
    $$
When $R$ is local and $P \in \Perf(R)$, we set $\beta_i(P) = \dim_k H_i(P \otimes_R k)$ and $\beta(P) = \sum_i \beta_i(P)$. 
If $P$ is minimal, then $\beta_i(P) = \rank_R(P_i)$ for all $i$. 
For a finitely generated $R$-module $M$ of finite projective dimension, we set $\beta (M) = \beta(P) = \sum_j \dim_k \Tor_j^R(M, k)$, where $P$ is any finite rank minimal free resolution of $M$.
\end{notation}

\subsection*{Acknowledgments} 

The authors thank Lucho Avramov for suggesting the argument used to prove Corollary \ref{MainCor} along with helpful comments and corrections on an earlier draft of this paper, and Linquan Ma for helpful discussions regarding Hochster's metatheorem. We thank the anonymous referees for catching errors and making suggestions that improved the quality of the exposition.

\section{Background}\label{sec:background}

\subsection{The Dold-Kan Correspondence}

In this section we recall details regarding one of the main tools of this paper, the \emph{Dold-Kan} correspondence. This correspondence gives a method for extending (not
necessarily additive) functors
defined at the level of $R$-modules to functors of complexes, in such a way that homotopy equivalences are preserved.
We will summarize and make explicit the necessary key properties of this correspondence, and we refer the reader to \cite{dold1961homologie,lurie2016higher} for further details.

\begin{defn}
    The \emph{simplicial category} $\Delta$ is the category whose objects are the finite, non-empty ordered sets $[n] := \{ 0 < 1 < \cdots < n \}$ for $n \geq 0$, with morphisms
    $$
    \Hom_\Delta ([n] , [m] ) := \{\text{nondecreasing set maps} \ [n] \to [m]\}.
    $$
    Given a category $\cC$, a \emph{simplicial object} in $\cC$ is a
    contravariant functor from $\Delta$ to $\cC$.
    The class of all simplicial objects in $\cC$ forms a category $\sC$ with morphisms given by natural transformations.
\end{defn}



Assume $\cC$ is an idempotent complete, exact category; i.e., a full subcategory of an abelian category that is closed
under extensions and summands. The example we have in mind occurs when
$\cC$ is the category of finitely generated and projective $R$-modules, for a commutative ring $R$. 
Then $\sC$ is also an exact category, with the notion of exactness given component-wise. 
Let $\Chzero(\cC)$ denote the category of non-negative chain complexes of objects in $\cC$ with morphisms being chain maps; it is also an
exact category with exactness defined degreewise.

\begin{thm}[Dold-Kan Correspondence, {\cite{dold1958homology}, \cite[Theorem 1.2.3.7]{lurie2016higher}}]    \label{thm:DoldKan}
   For $\cC$ as above, there is an equivalence of categories 
   $$
   \xymatrix{
     \sC \ar@/^/[rr]^{N}_\cong && \Chzero(\cC). \ar@/^/[ll]^{\G}
   }
   $$
   %
%
Moreover,
\begin{itemize}
\item $N$ converts simplicial homotopies to chain homotopies,
\item $\G$ converts chain homotopies to simplical homotopies,
\item   $N$ and $\G$ are exact functors, and
\item $N$ and $\G$ are natural for additive functors $\cC \to \cC'$.
\end{itemize}
\end{thm}

    The functor $N$ is the \emph{normalization} functor, which associates
    to a simplicial object a chain complex whose degree $n$ piece is the intersection of the kernels of the first $n$ face maps.
    For a complex
    \begin{equation} \label{E418}
      P = ( \cdots \xra{\del_2}  P_1 \xra{\del_1} P_0 \to 0),
    \end{equation}
    the degree $n$ component of the simplicial object $\G(P)$ is $\bigoplus_{[n] \onto [k]} P_k$ where the direct sum is indexed by surjections in $\Delta$ with
    source $[n]$. For instance $\G(P)_0 = P_0$, $\G(P)_1 = P_0 \oplus P_1$, the degeneracy $\G(P)_0 \to \G(P)_1$ is the evident inclusion and the two face maps $\G(P)_1 \to
    \G(P)_0$ are $d_0(p_0, p_1) = p_0$ and $d_1(p_0, p_1) = p_0 + \del_1(p_1)$.
    

\begin{defn}
  Let $\cC$ be an idempotent complete additive category.
  Given a covariant (not necessarily additive) endofunctor $G : \cC \to \cC$ define $G_*: \sC \to \sC$ be the functor given by composition: $G_*$ sends 
 a simplicial object  $M: \Delta^\op \to \cC$ to the simplicial object $G \circ M: \Delta^\op \to \cC$. 
The {\em Dold-Kan extension} of $G$ is the functor  $\wt{G}: \Chzero(\cC) \to \Chzero (\cC)$ 
  given by
  $$
  \wt{G} := N \circ G_* \circ \Gamma.
  $$
\end{defn}

For example, given a complex $P$ as in \eqref{E418} and functor $G$, the complex $\tilde{G}(P)$
has $\tilde{G}(P)_0 = G(P_0)$ and $\tilde{G}(P)_1 = \ker(G(P_0 \oplus P_1) \xra{G(d_0)} G(P_0))$ 
and the differential $\tilde{G}(P)_1 \to \tilde{G}(P)_0$ is the restriction of $G(d_1)$.

\begin{cor} \label{DKhomotopy}
The Dold-Kan extension of any functor preserves chain homotopies.
\end{cor}

\begin{defn}
  An endofunctor $G: \cC \to \cC$ is called {\em polynomial} if $G(0) = 0$ and, for some $d \geq 0$, the $s$-th cross effects functor $T_s(G): \cC^{\oplus s} \to \cC$ 
  is zero for all $s > d$ (see \cite[\S 9]{EM54}). 
 The minimum value of $d$ for which this holds is the {\em degree} of $G$.
 \end{defn}

 A functor is additive if and only if it is polynomial of degree
  at most $1$ \cite[9.11]{EM54}. (A functor has degree $0$ if and only if it is the zero functor.)
  An example of a non-additive polynomial functor is the second exterior power functor $\L^2_R$ defined on the category of
  finitely generated, projective $R$-modules (see below). 
  We have a  natural isomorphism
\begin{equation} \label{E418c}
  \L^2_R(P \oplus P') \cong \L^2_R(P) \oplus P \otimes_R P' \oplus \L_R^2(P'),
\end{equation}
which shows that the second cross effects functor $T_2$  associated to $\L^2_R$ is given by $T_2(P, P') = P \otimes_R P'$. Since $T_2$ is additive in each
argument, we have $T_s = 0$ for all $s > 2$ \cite[9.9]{EM54}.
The same holds for the second symmetric power functor $S^2_R$ and the second tensor power functor $T^2_R$; see below.

The following is given by \cite[4.7 and Hilfssatz 4.23]{dold1961homologie}.

\begin{prop} \label{PolyFunctorProp}
  If $G$  is a polynomial functor of degree $d$ and $P \in \Chzero(\cC)$ is concentrated in degrees $[0, n]$, i.e., 
  $P = (\cdots \to 0 \to P_n \xra{d_n}  \cdots \xra{d_2}   P_1  \xra{d_1}   P_0 \to 0)$, then the complex
  $\wt{G}(P)$ is concentrated in degrees $[0, d n]$.  
  If $G$ is additive, then there is a natural isomorphism
  $$
  \wt{G}(P) \cong G(P) := (\cdots \to 0 \to G(P_n)  \xra{G(d_n)} \cdots  \xra{G(d_2)} G(P_1)  \xra{G(d_1)} G(P_0) \to 0).
  $$
\end{prop}

We now focus on our primary case of interest:
Let $R$ be a commutative ring and $\cP(R)$ the category of finitely generated and projective $R$-modules. 

\begin{defn} An endo-functor $G: \cP(R) \to \cP(R)$ {\em commutes with localization} if for each prime ideal $\fp$ of $R$ there exists an endofunctor
  $G_{\fp}: \cP(R_\fp) \to \cP(R_\fp)$ and a natural isomorphism that joins the functors $P \mapsto G(P)_\fp$ and $P \mapsto G_{\fp}(P_\fp)$.
\end{defn}

Observe that if $G$ commutes with localization then the naturality of the Dold-Kan correspondence gives a natural isomorphism
\begin{equation} \label{E226}
  \wt{G_\fp}(P_\fp) \cong \wt{G}(P)_\fp
\end{equation}
in $\Chzero(\cP(R_\fp))$ for each $P \in \Chzero(\cP(R))$.

Recall that the {\em support} of a complex $P$ of $R$-modules is
$$
\supp_R(P) = \{ \fp \in \Spec(R) \mid \text{ $P_\fp$ is not exact}\}.
$$
If $P$ is a bounded below complex of projective $R$-modules, then $\fp \in \supp_R(P)$ if and only if $P_{\fp}$ is not contractible
as a complex of $R_\fp$-modules.

\begin{prop} \label{PropCWL} If $G: \cP(R) \to \cP(R)$ commutes with localization, then $\supp_R(\wt{G}(P)) \subseteq \supp_R(P)$
  for all $P \in \Chzero(\cP(R))$. 
\end{prop}

\begin{proof} 
If $\fp \notin \supp_R(P)$ then $P_\fp$ is contractible (that is, the identity map is homotopic to the zero map) and hence $\wt{G_\fp}(P_\fp)$ is also contractible by Corollary \ref{DKhomotopy}.
The result thus follows from \eqref{E226}.
\end{proof}

We write $\Perf_{\geq 0}(R)$ and $\Perf^\fl_{\geq 0}(R)$ for the full subcategories of $\Perf(R)$ and $\Perf^\fl(R)$ (see \S \ref{sec:notation})
consisting of complexes with $P_j = 0$ for $j < 0$.
By Proposition \ref{PolyFunctorProp}, given any polynomial endofunctor $G$ of $\cP(R)$, 
its Dold-Kan extension determines an endofunctor $\wt{G}$ of $\Perf_{\geq 0}(R)$ that preserves chain homotopies. Since the mapping cone of a quasi-isomorphism $\phi : C \to D$ of objects in $\Perf_{\geq 0} (R)$ is an exact bounded complex of projective $R$-modules, the map $\phi$ must be a homotopy equivalence; this implies that the Dold-Kan extension of a functor also preserves quasi-isomorphisms on the category $\Perf_{\geq 0} (R)$. If $G$ commutes with localization,
then $\wt{G}$ restricts to an endofunctor of $\Perf^\fl_{\geq 0}(R)$.

\begin{rem}
  The work of Tchernev-Weyman \cite{tchernev2004free} provides an alternative method for extending polynomial functors to arbitrary complexes of free $R$-modules.
  Their construction is distinct from the Dold-Kan extension in general (see \cite[Example 14.7]{tchernev2004free}).
\end{rem}

We will apply the constructions and results above in the following cases:

\begin{itemize}
\item $T^2_R = T^2: \cP(R) \to \cP(R)$, the second tensor power functor, given by $T^2(P) = P \otimes_R P$.
\item $S^2_R = S^2: \cP(R) \to \cP(R)$, the second symmetric power functor, given by
  $$
  S^2 P = \frac{P \otimes_R P}{R \cdot \{p \otimes p' - p' \otimes p \mid p, p' \in P\}}.
  $$
\item $\L^2_R = \L^2: \cP(R) \to \cP(R)$, the second exterior  power functor, given by
  $$
  \L^2(P) = \frac{P \otimes_R P}{R \cdot \{p \otimes p \mid p \in P\}}.
  $$
    \item $\chr(R) = p > 0$ and $F: \cP(R) \to \cP(R)$ is extension of scalars along the Frobenius homomorphism $\phi: R \to R$ given by  $\phi(r) = r^p$.
    \end{itemize}

    Each is a polynomial functor of degree at most $2$ that commutes with localization, and thus each induces a quasi-isomorphism preserving endo-functor of $\Perf^\fl_{\geq 0}(R)$, written as
    $\wt{T^2}$, $\wt{S^2}$, $\wt{\L^2}$, and $\wt{F}$, respectively.  

    Since $F$ is additive, $\wt{F}$ is isomorphic to the usual functor on chain complexes given by extension of scalars along the Frobenius (see Proposition \ref{PolyFunctorProp}).
    Likewise, the functor $\wt{T^2}$ can be related to the ``usual'' notion of tensor product of complexes (see, e.g.,  \cite[Section 2.3]{schwede2003equivalences}):
    
    \begin{prop} \label{prop228}
      For any $P \in \Perf_{\geq 0}(R)$, there is a homotopy equivalence joining $\wt{T^2}(P)$ and $P \otimes_R P$ (the totalization of the
      bicomplex which in bidegree $(s,t)$ is $P_s \otimes_R P_t$).
    \end{prop}

    \begin{rem} Analogous statements hold for $\wt{S^2}$ and $\wt{\L^2}$ for rings $R$ in which $2$ is invertible.
    \end{rem}

    We will also need the following fact:

    \begin{lem} \label{lem419} For $P \in \Chzero(\cP(R))$ there is an isomorphism $H_0(\widetilde{\L^2}(P)) \cong \L^2(H_0(P))$.
    \end{lem}

    \begin{proof} 
Based on the description of the beginning of the complex $\widetilde{\L^2}(P)$ given above and using \eqref{E418c}, there is a right exact sequence
      $$
      P_0 \otimes_R P_1 \oplus \L^2(P_1) \to \L^2(P_0) \to H_0(\widetilde{\L^2}(P)) \to 0
      $$
      in which the left-hand map sends $(p_0 \otimes p_1, p_1' \smsh p_1'')$ to $p_0 \smsh \del_1(p_1) + \del_1(p_1') \smsh \del_1(p_1'')$.
      It follows that there is an isomorphism
      $$
      H_0(\widetilde{\L^2}(P)) \cong
      \coker \left(P_0 \otimes P_1 \xra{p_0 \otimes p_1 \mapsto p_0 \smsh \del_1(p_1)} \L^2(P_0)\right).
      $$
      This cokernel also maps isomorphically to $\L^2 H_0(P)$ via the map induced by the canonical map  $P_0 \onto H_0(P)$.
    \end{proof}

The following lemma will be a key in our proof of the Total Rank Conjecture in characteristic two:
    
\begin{lem}\label{lem:fundamentalSES}
    Let $R$ be a noetherian ring and $P \in \Perf_{\geq 0} (R)$. There is a short exact sequence of complexes in $\Perf_{\geq 0} (R)$ of the form
    $$
    0 \to \wt{\L^2} P \to \wt{T^2} P \to \wt{S^2} P \to 0
    $$
    and this induces a long exact sequence 
      \begin{equation} \label{E419a}
\begin{aligned}
 \cdots  & \to H_{j+1}(\wt{S^2} P) \to H_j(\wt{\L^2} P) \to H_j(P \otimes_RP) \to H_j(\wt{S^2} P) \to \\
  &  \cdots \to
H_{1}(\wt{S^2} P) \to  H_0(\wt{\L^2} P) \to  H_0(P \otimes_R P) \to H_0(\wt{S^2} P) \to 0. \\
  \end{aligned}
   \end{equation}
    When $\chr(R) = 2$, there is a short exact sequence of complexes of the form
    $$
    0 \to F(P) \to \wt{S^2} P \to \wt{\L^2} P \to 0
    $$
    and hence a long exact sequence
      \begin{equation} \label{E419b}
\begin{aligned}
  \cdots & \to  H_j(F(P)) \to H_j(\wt{S^2} P) \to H_j(\wt{\L^2} P) \to H_{j-1}(F(P)) \to \\
  & \cdots \to H_1(\wt{S^2} P) \to H_1(\wt{\L^2} P) \to H_0(F(P)) \to H_0(\wt{S^2} P) \to H_0(\wt{\L^2} P) \to 0. \\
  \end{aligned}
   \end{equation}

    \end{lem}

\begin{rem}
  Taking classes in the Grothendieck group of $\Perf(R)$, the second short exact sequence of Lemma \ref{lem:fundamentalSES}
gives an explicit proof that extension of scalars along the Frobenius map realizes the second Adams operation; see \cite[Theorem B]{gillet1987intersection}. 
\end{rem}

\begin{proof}
  For any finitely generated projective $R$ module $P$, we have the natural short exact sequence
  $$
  0 \to \L^2(P) \xra{x \smsh y \mapsto x \otimes y - y \otimes x} T^2(P) \xra{\can} S^2(P) \to 0
  $$
  and when $\chr(R) = 2$ we have the natural short exact sequence
  $$
  0 \to F(P) \to S^2(P) \xra{x \cdot y \mapsto x \smsh y} \L^2(P) \to 0. 
  $$
Here, the map $F(P) \to S^2(P)$ is adjoint to the $R$-linear map $P \xra{x \mapsto x^2} \phi_* S^2(P)$, where $\phi_*$ denotes restriction of scalars along the Frobenius. 
The result follows from the naturality and exactness properties of Dold-Kan extensions, the isomorphism $\wt{F}(P) \cong F(P)$, and the quasi-isomorphism
$\wt{T^2}(P) \sim P \otimes_R P$. 
  \end{proof}

\subsection{Some generalities on Dutta Multiplicities and Short Complexes}

Throughout this section, we assume that $(R, \fm, k)$ is a complete local ring of characteristic $p > 0$ and dimension $d$ that has a perfect residue field $k = R/\fm$.
We recall properties of extension of scalars along the Frobenius map when applied
to free complexes with finite length homology. In particular, we recall the definition of ``Dutta multiplicity".

Given a complex of free  $R$-modules $P$, recall that we write
$F(P)$ for the complex of $R$-modules obtained from $P$ by extension of scalars along the Frobenius endomorphism $r \mapsto r^{p}$ of $R$. 
For an integer $e \geq 0$, we write $F^e(P)$ for the $e$-th iterate of this functor.
Concretely, if bases are chosen for each component of $P$, so that each map in the complex is represented by a matrix, then $F^e(P)$ is isomorphic to the complex
whose component free modules are the same as those of $P$, but with each entry of each matrix replaced by its $p^e$ power.

The following result follows from \cite[Theorem 7.3.3]{RobertsBook}:

\begin{prop} \label{proplimitsexist} For $R$ as above, given $P \in \Perf^\fl(R)$ and an integer $i$, the limit
  $\lim_{e \to \infty} \frac{h_i(F^eP)}{p^{de}}$
  exists.
\end{prop}

\begin{defn} With $R$ as above and  any $P \in \Perf^\fl(R)$, define 
  $$
  h_i^\infty(P) = \lim_{e \to \infty} \frac{h_i(F^eP)}{p^{de}} \and
  h^\infty(P) = \sum_i h_i^\infty(P). 
  $$
The 
  {\em Dutta multiplicity} of $P$ is
$$
\chi^\infty(P) = \sum_i (-1)^i h^\infty_i(P) = \lim_{e \to \infty} \frac{\chi(F^e P)}{p^{de}}.
$$
\end{defn}

Recall from the introduction that a {\em tiny complex} is a complex of finite rank free $R$-modules of the form $0 \to P_d \to \cdots P_1 \to P_0 \to 0$
such that $H(P)$ is nonzero and of finite length.
The following result is due to Roberts; see the proof of \cite[7.3.5]{RobertsBook}.
It says that tiny complexes behave like resolutions of modules ``in the limit". 

\begin{rem}
   The philosophy that ``tiny complexes behave like resolutions" is also exploited in the work \cite{iyengar2022multiplicities}.
\end{rem}

\begin{thm}[{\cite[Theorem 7.3.5]{RobertsBook}}]\label{thmRoberts}
  Let $(R, \fm, k)$ be a complete local ring of characteristic $p > 0$ and dimension $d$ with a perfect residue field.
  If $P$ is a tiny complex, then $h_i^\infty(P) = 0$ for all $i \ne 0$ and $h_0^\infty(P) > 0$.
  In particular,
  $$
h^\infty(P) = \chi^\infty(P) > 0.
$$
\end{thm}

\begin{rem} \label{rem1222}
  If $P$ has the form $\cdots \to 0 \to P_d \to \cdots \to P_1 \xra{d_1} P_0 \to 0 \to \cdots$ with $P_1 = R^n$, $P_0 = R$, and $d_1 = (a_1, \dots, a_n)$
  then
$$
\chi^\infty(P) = \lim_{e \to \infty} \frac{\len_R(R/(a_1^{p^e}, \dots, a_n^{p^e}))}{p^{de}}, 
$$
which is by definition  the {\em Hilbert-Kunz multiplicity} of the ideal $(a_1, \dots, a_n)$.
\end{rem}

\begin{rem} When $R$ is a Roberts ring --- e.g., a complete intersection --- we have $\chi^\infty(P) = \chi(P)$ for any tiny complex $P$.
\end{rem}

In fact, a straightforward variant of the argument used by Roberts in the proof of Theorem \ref{thmRoberts} shows a bit more:

\begin{thm}[{\cite[Theorem 7.3.5]{RobertsBook}, \cite[Proposition 1.3]{dutta1996ext}}]  \label{thmRoberts2}
  With the same assumptions on $R$ as in Theorem \ref{thmRoberts},
  if $P$ is a complex of finite rank free $R$-modules of the form
  $$
  0 \to P_m \to \cdots \to P_1 \to P_0 \to 0
  $$
that has finite length homology, then $h_i^\infty(P) =  0$ for all $i > m- d$. 
\end{thm}

For the sake of completeness, we include Roberts' proof:

\begin{proof} 
  Let $I = (0 \to I^0 \to \cdots \to I^d \to 0)$ be the dualizing complex of $R$. This complex has the following properties:
  \begin{itemize}
  \item $H^j(I)$ is a finitely generated $R$-module of dimension at most $d-j$ for all $0 \leq j \leq d$.
  \item For each $j$ we have
    $$ I^j = \bigoplus_{\fp \in \Spec(R), \dm(R/\fp) = d-j} E_R(R/\fp);
    $$
    in particular,  $I^d = E_R(k)$, the injective hull of the residue field.
  \end{itemize}
By Grothendieck duality and the fact that $F^e(P)$ has finite length homology for all $e \geq 0$,
  there is a quasi-isomorphism  $\Hom_R(F^e(P), I) \sim \Sigma^{-d} \Hom(F^e(P), E(k))$.
  By Matlis Duality, it follows that $h^j \Hom_R(F^e(P), I) = h_{j-d}(F^e(P))$. Taking limits yields the equality
  \begin{equation} \label{E311}
  h^\infty_{j-d}(P) = \lim_{e \to \infty}  \frac{h^j \Hom_R(F^e(P), I)}{p^{de}}.
\end{equation}

  On the other hand, the spectral sequence
  $$
  E^{p,q}_2 = H^p\Hom_R(F^e(P), H^q(I)) \Longrightarrow H^{p+q} \Hom_R(F^e(P), I)
  $$
  gives
  $$
  h^j \Hom_R(F^e(P), I)) \leq \sum_{p+q = j} h^p \Hom_R(F^e(P), H^q(I)),
  $$
  for each $e \geq 0$. 
  Since  $H^q(I)$ is finitely generated and of dimension $d-q$, we have
  $$
  \lim_{e \to \infty} \frac{h^p\Hom_R(F^e(P), H^q(I))}{p^{de}} = 0
  $$
  for all $q > 0$ and all $p$; see  \cite[Proof of Theorem 7.3.5]{RobertsBook}. It follows that
  $$
  h^\infty_{j-d}(P) =  \lim_{e \to \infty} \frac{h^j \Hom_R(F^e(P), I)}{p^{de}} 
  \leq 
  \lim_{e \to \infty} \frac{h^j\Hom_R(F^e(P), H^0(I))}{p^{de}},
  $$
  where the equality is \eqref{E311}. 
  The result follows since $H^j\Hom_R(F^e(P), H^0(I)) = 0$ whenever $j > m$.
\end{proof}

\section{The Total Rank Conjecture for Roberts Rings in Characteristic 2}\label{sec:TotalRankRobertsChar2}

We initially prove the Total Rank Conjecture for a special class of rings:

\begin{defn} A local ring $(R, \fm)$ with $\chr(R) = 2$ is called {\em quasi-Roberts} if it is Cohen-Macaulay and for each
  $R$-module $M$ of finite length and finite projective dimension we have $\len(F(M)) = 2^{\dm(R)} \cdot \len(M)$, where $F$ denotes extension of scalars along the Frobenius map.
\end{defn}

Examples of quasi-Roberts rings include all complete intersections (see \cite[Proposition 7.1]{gillet1987intersection}). There exist Gorenstein rings that are not quasi-Roberts; for instance, the work of Miller--Singh \cite[Section 6]{miller2000intersection} constructs a $5$-dimensional Gorenstein ring $R$ in characteristic $2$ and a module $M$ of finite length over $R$ with $\ell (M) = 222$ but
$$\ell (F^2 (M)) = 2^5 \cdot \left( 220 + \frac{1}{2} \right) < 2^5 \ell(M).$$

\begin{thm} \label{thmA}
 Suppose $(R, \fm)$ is a local ring of characteristic $2$ and dimension $d$ that is quasi-Roberts. For any non-zero $R$-module $M$ of finite length and finite projective dimension, we have
  $$
  \beta(M) \geq 2^d+ \frac{2 \len_R(\Lambda^2_R(M))}{\len_R(M)} \geq 2^d.
  $$
  Moreover, if $\beta(M) = 2^{d}$ then $M \cong R/(f_1, \dots, f_d)$ for a regular sequence of elements $f_1, \dots, f_d \in \fm$.
\end{thm}

\begin{proof} Initially, let $R$ be any local ring of characteristic $2$ and $P$  any object of $\Perf^\fl_{\geq 0}(R)$. 

  The long exact sequence \eqref{E419a} of Lemma \ref{lem:fundamentalSES}
gives a short exact sequence
{\small
  $$
0 \to  \coker\left(H_{j+1}(\wt{S^2}(P)) \to H_{j}(\wt{\L^2}(P))\right) \to H_j(P \otimes_R P) \to
 \ker\left(H_{j}(\wt{S^2}(P)) \to H_{j-1}(\wt{\L^2}(P))\right)  \to 0
  $$
  }for each $j \geq 0$, from which  it follows that
\begin{equation} \label{E310a}
h(P \otimes_R P) \geq h_0(\wt{S^2}(P)) + \sum_{j\geq 0} \left|h_j(\wt{\L^2} (P)) - h_{j+1}(\wt{S^2}(P))\right|.
\end{equation}

Let us now assume that $R$ is Cohen-Macaulay and $P$ is a tiny complex. Then $P$ is a resolution of $M := H_0(P)$
and  $F(P)$ is a resolution of $F(M)$; in particular, 
$H_j(F(P)) = 0$ for $j \ne 0$. The long exact sequence \eqref{E419b} of Lemma \ref{lem:fundamentalSES} thus gives the equations
\begin{equation} \label{E310b}
  h_j(\wt{S^2}(P)) = h_j(\wt{\L^2}(P)), \text{ for all $j \geq 2$,}
  \end{equation}
  and the exact sequence
$$
  0 \to H_1(\wt{S^2}(P)) \to H_1(\wt{\L^2}(P))  \to F(M) \to
  H_0(\wt{S^2}(P)) \to H_0(\wt{\L^2}(P))  \to 0.
$$
The latter yields
  \begin{equation} \label{E310c}
  -  h_1(\wt{S^2}(P)) + h_1(\wt{\L^2}(P)) 
  +   h_0{\wt{S^2}(P)} - h_0{\wt{\L^2}(P)} = \len_R(F(M)).
\end{equation}
Combining \eqref{E310a} and \eqref{E310b} gives
$$
h(P \otimes_R P) \geq h_0(\wt{S^2}(P)) +  h_0(\wt{\L^2}(P)) - h_1(\wt{S^2}(P)) + h_1(\wt{\L^2}(P))
$$
and then applying \eqref{E310c} and using that $H_0(\wt{\L^2}(P)) \cong \L^2_R(M)$ by Lemma \ref{lem419}, 
we obtain
\begin{equation} \label{E310z}
h(P \otimes_R P)     \geq \len_R(F(M)) + 2 \len_R (\L^2_R(M)).
\end{equation}

Now assume that $R$ is quasi-Roberts, so that $\len_R(F(M)) = 2^d \len_R(M)$.
Then \eqref{E310z} gives 
$$
h(P \otimes_R P)     \geq 2^d \len_R(M) + 2 \len_R (\L^2_R(M)).
$$
Using the quasi-isomorphism $P \otimes_R P \sim M \otimes_R P$, we obtain
\begin{equation} \label{E310f}
  \len_R(M) \cdot \beta(M) \geq h(P \otimes_R P)
\end{equation}
with equality holding if and only if the differential on $P \otimes_R M$ is $0$. The first result follows by dividing through by $\len_R(M) > 0$. 

We can prove the final claim via an argument that is nearly identical to that found in \cite[2.4]{walker2017total}; in more detail:
If $\beta(M) = 2^d$ then $\L^2_R(M) = 0$ and the inequality in \eqref{E310f} must be an equality.
The first condition means $M \cong R/I$ for some ideal $I$ of finite projective dimension, and the latter condition implies that the complex $R/I \otimes_R P$ has a trivial differential.
But then $I/I^2 \cong \Tor_1^R(R/I, R/I)$ is a free $R/I$-module, which, together with the fact that $\pd_R(I) < \infty$, implies that $I$ is generated by a regular sequence by
a theorem due to Ferrand and Vasconcelos \cite[2.2.8]{bruns1998cohen}.
\end{proof}

For a general Cohen-Macaulay ring $R$ and finite length $R$-module $M$,  we do not know whether $\beta(M) = 2^d$ implies that $M$ must be
  isomorphic to the quotient of $R$ by a regular sequence. It seems reasonable to conjecture that this is so;  see Conjecture \ref{ExtendedConjecture} in the Introduction.

\section{The Total Rank Conjecture in Characteristic 2}\label{sec:genTRC}

The goal of this section is to prove the following:

\begin{thm} \label{thmB}
  Let $(R, \fm, k)$ be a local ring of characteristic $2$ and dimension $d$. If  $P$ is a finite rank free complex of the form
  $$
  0 \to P_{d+1} \to \cdots \to P_1 \to P_0 \to 0
  $$
  having finite length, nonzero homology, then $\sum_j \rank_R(P_j) \geq \beta(P) \geq 2^d$.
\end{thm}

\begin{rem} As noted in the introduction, the previous Theorem is false without the assumption on the characteristic, even for regular rings.  Indeed,
 if $(R, \fm, k)$ is any  local ring of dimension $d \geq 8$ such that  $\chr(k) \ne 2$, then as shown in \cite{iyengar2018examples} there exists a complex of the form in the Theorem
  with $\beta(P) < 2^d$.
\end{rem}

\begin{cor} \label{MainCor}
  Suppose $(R, \fm, k)$ is a local ring of dimension $d$ such that $\chr(k) = 2$ and let $P$ be a tiny complex over $R$.
Then $\beta(P) \geq 2^{\dm(R/2R)} \geq 2^{d-1}$; in particular, if $2$ is nilpotent in $R$ (e.g., if $\chr(R) = 2$), then $\beta(P) \geq 2^d$.  
\end{cor}

\begin{proof} 
  Since $\rank_R(P) = \rank_{R/2}(P/2P)$ and $\dm(R/2) \in \{\dm(R), \dm(R) - 1\}$, this follows from the 
  Theorem applied to the complex $P/2P$ of free $R/2$-modules.
\end{proof}

The proof of Theorem \ref{thmB} uses the following lemma.
  
\begin{lem}\label{lem:tensorIneq}
    Assume $R$ is a complete local ring of characteristic $p > 0$ with a perfect residue field, and let $P \in \Perf^\fl (R)$. Then there is an inequality
    $$
    h^\infty(P \otimes_R P) \leq h^\infty(P) \beta (P).
    $$
\end{lem}

\begin{proof} For each $e$, we have
  $$
  h(F^e(P \otimes_R P)) = h(F^e(P) \otimes_R F^e(P)) \leq h(F^e(P)) \beta(F^e P) =  h(F^e(P)) \beta(P),
  $$
  where the inequality follows from the spectral sequence
  $$
  E^{p,q}_2 = H^p(H^q(F^e(P)) \otimes_R F^e(P)) \Longrightarrow H^{p+q}(F^e(P) \otimes_R F^e (P)).
  $$
  The result follows by taking suitable limits, using Proposition \ref{proplimitsexist}.  
\end{proof}

\begin{proof}[Proof of Theorem \ref{thmB}]
As in \cite[Proof of Theorem 2 part (2)]{walker2017total}, we may reduce to the case where $R$ is complete and has an algebraically closed residue field.

For any finite rank free $R$-module $E$, we have a natural isomorphism $F(G(E)) \cong G(F(E))$ where $G = T^2, S^2$, or $\L^2$ (since $F$ commutes with tensor constructions).
  By the naturality of Dold-Kan extensions, we obtain the isomorphisms
  \begin{equation} \label{E314a}
    \begin{aligned}
      F^e(\wt{T^2}(P)) & \cong \wt{T^2}(F^e(P)) \\
  F^e(\wt{S^2}(P)) & \cong \wt{S^2}(F^e(P)) \\
  F^e(\wt{\L^2}(P)) & \cong \wt{\L^2}(F^e(P)) \\
  \end{aligned}
\end{equation}
for each $P \in \Perf_{\geq 0}(R)$ and each $e \geq 0$.  
  
By \eqref{E419b}, for all $e \geq 0$ and $j$ there is an exact sequence
  $$
  H_j(F^{e+1}P) \to H_j(\wt{S^2}(F^e P)) \to H_j(\wt{\L^2}(F^eP))  \to H_{j-1}(F^{e+1}(P)).
  $$
  Using \eqref{E314a} and Theorem \ref{thmRoberts2}, upon taking suitable limits we obtain the equality
  \begin{equation} \label{E311c}
  h_j^\infty(\wt{S^2}(P)) = h_j^\infty(\wt{\L^2}(P)), \, \text{  for all $j \geq 3$.}
\end{equation}
We also have the exact sequence
  \begin{equation} \label{E314b}
  0 \to L_e \to H_2(\wt{S^2}(F^eP)) \to H_2(\wt{\L^2}(F^eP))  \to H_1(F^{e+1}(P))
  \to  H_1(\wt{S^2}(F^eP)) 
  \end{equation}
where $L_e := \im\left(H_2(F^{e+1}(P)) \to  H_2(\wt{S^2}(F^eP))\right)$. 
Using Theorem \ref{thmRoberts2} again, we have $\lim_{e \to \infty} \frac{\len_R L_e}{2^{de}} = 0$, and thus by using \eqref{E314a} and
taking suitable limits, the exact sequence \eqref{E314b} yields
$$
h_1^\infty(F(P)) \leq h_1^\infty(\wt{S^2}(P)) + h_2^\infty(\wt{\L^2}(P)) - h_2^\infty(\wt{S^2}(P)).
$$
Similarly, the exact sequence 
  $$
  H_1(\wt{\L^2}(F^eP))  \to H_0(F^{e+1}(P))  \to H_0(\wt{S^2}(F^eP)) \to H_0(\wt{\L^2}(F^eP))  \to 0
  $$
  gives
  $$
 h_0^\infty(F(P)) \leq h_1^\infty(\wt{\L^2}(P)) + h_0^\infty(\wt{S^2}(P)) - h_0^\infty(\wt{\L^2}(P)).
$$
It is evident from the definition of $h^\infty$ that $h^\infty(F(P)) = 2^d h^\infty(P)$, and thus
\begin{equation} \label{E311e}
  2^d  h^\infty(P) \leq
h_1^\infty(\wt{\L^2}(P)) + h_0^\infty(\wt{S^2}(P)) - h_0^\infty(\wt{\L^2}(P)) +
h_1^\infty(\wt{S^2}(P)) + h_2^\infty(\wt{\L^2}(P)) - h_2^\infty(\wt{S^2}(P)).
\end{equation}

Applying \eqref{E310a} to $F^e(P)$ for each $e \geq 0$, using \eqref{E314a} and taking suitable limits, we find that there is an inequality
\begin{equation} \label{E314c}
  h^\infty(P \otimes_R P) \geq h^\infty_0(\wt{S^2}(P)) + \sum_{j \ge 0} |h^\infty_j(\wt{\L^2}(P)) -  h^\infty_{j+1}(\wt{S^2}(P))|.
  \end{equation}
Using this along with \eqref{E311c} gives
  $$
  h^\infty(P \otimes_R P) \geq h_0(\wt{S^2}(P))  - h^\infty_0(\wt{\L^2}(P)) +  h^\infty_{1}(\wt{\L^2}(P))
+h^\infty_1(\wt{S^2}(P)) -  h^\infty_{2}(\wt{\L^2}(P)) + h^\infty_2(\wt{S^2}(P)).
$$
Combining this with  \eqref{E311e} results in 
$$
h^\infty(P \otimes_R P) 
\geq 2^d h^\infty(P)
$$
so that by Lemma \ref{lem:tensorIneq} we arrive at 
$$
h^\infty(P) \beta(P) \geq 2^d h^\infty(P).
$$
To complete the proof, we just need to show $h^\infty(P) > 0$.
If $H_0(P) \ne 0$, then
$$
h_0^\infty(P) = \lim_{e \to \infty} \frac{\len_R(F^e H_0(P)}{2^{ed}} \geq
\lim_{e \to \infty} \frac{\len_R F^e(k)}{2^{ed}} = e_{HK}(R) > 0,
$$
where $e_{HK}$ denotes the Hilbert-Kunz multiplicity, which is known to be positive (see \cite[Chapter 7.3]{RobertsBook}). If $H_0(P) = 0$ then $P$ is quasi-isomorphic to the suspension of a tiny complex $P'$. By the New Intersection Theorem, we must have $H_0(P') \ne 0$,
so that $h_1^\infty(P) = h_0^\infty(P') > 0$ by  the previous case. 
\end{proof}

Theorem \ref{thmB} raises the question of what can be said about the total Betti number of an arbitrary object $P \in\Perf^\fl(R)$ for rings of characteristic $2$.
Unlike the case where $\chr(k) \ne 2$, we know of no counterexample to the conjecture that $\beta(P) \geq 2^{\dm(R)}$ for any such $P$.
We close with a result concerning complexes that are ``two terms longer'' than tiny complexes, which we view as
giving additional evidence that $\beta(P) \geq 2^{\dm(R)}$ might hold in general:

  \begin{thm}\label{thm:012homology}
  Suppose $R$ is a complete local ring of characteristic $2$ and dimension $d$ with a perfect residue field, and assume
  $$
  P = (0 \to P_{d+2} \to \cdots \to P_1 \to P_0 \to 0)
  $$
  is an object of $\Perf^\fl(R)$ concentrated in degrees $[0, d +2]$.
Then there is an inequality
$$\beta (P) \geq 2^{d} \frac{h^\infty_1 (P)}{h^\infty(P)}.$$
In particular, if
 $\chi^\infty(P) = 0$, then
  $\beta(P) \geq 2^{d-1}$.
\end{thm}

\begin{proof}
  By Theorem \ref{thmRoberts2} we have $h_j^\infty(F^eP) = 0$ for $j \notin \{0,1,2\}$ and all $e \geq 0$,
  and thus $h_j^\infty(\wt{\L^2}(P)) = h_j^\infty(\wt{S^2}(P))$ for $j \geq 4$.
Using \eqref{E314c} this gives
$$
\begin{aligned}
  h^\infty(P \otimes_R P) &   \geq  h^\infty_0(\wt{S^2}(P)) -   h^\infty_0(\wt{\L^2}(P)) + h^\infty_1(\wt{S^2}(P))   + |h^\infty_1(\wt{\L^2}(P)) 
  - h^\infty_2(\wt{S^2}(P))| \\
  & \phantom{XXXXXXXXXXXXX} + h^\infty_2(\wt{\L^2}(P)) - h^\infty_3(\wt{S^2}(P)) + h^\infty_3(\wt{\L^2}(P))  \\
& \geq  h^\infty_1(\wt{S^2}(P))+   |h^\infty_1(\wt{\L^2}(P)) - h^\infty_2(\wt{S^2}(P))| + h^\infty_2(\wt{\L^2}(P))  \\
& \geq  h^\infty_1(\wt{S^2}(P)) + h^\infty_2(\wt{\L^2}(P))   \\
\end{aligned}
$$
where the second inequality uses the fact that there is a surjection
$H_0(\wt{S^2}(F^eP)) \onto H_0(\wt{\L^2}(F^eP))$ 
and an injection $H_3(\wt{S^2}(F^eP) ) \into H_3(\wt{\L^2}(F^eP))$ for all $e \geq 0$. 

Now, the exact sequence  $H_2(\wt{\L^2}(F^eP)) \to H_1(F^{e+1}) \to H_1(\wt{S^2}(F^e(P))$ for each $e \geq 0$ implies that
$$
h_1^\infty(F(P)) \leq  h^\infty_1(\wt{S^2}(P)) + h^\infty_2(\wt{\L^2}(P)).
$$
Since $h_1^\infty(F(P)) = 2^d h_1^\infty(P)$,  
by using Lemma \ref{lem:tensorIneq} we arrive at
$$
h^\infty(P) \beta(P) \geq 2^d h_1^\infty(P).
$$
Finally, $h^\infty(P) > 0$ by an argument similar to that in the proof of Theorem \ref{thmB}, and the first assertion follows.
If $\chi^\infty(P) = 0$, then since $h_j^\infty(P) = 0$ for $j \notin \{0,1,2\}$,
we have $h^\infty(P) = 2 h_1^\infty(P)$ and the second assertion follows.
\end{proof}

\section{The Total Rank Conjecture for Short Complexes for Algebras over a Field}\label{sec:TRCOddChar}

The goal of this section is to show that $\beta (P) \geq 2^{\dm(R)}$ holds for any short complex over a $k$-algebra, where $k$ is a field.
The previous section has already proved this for rings of characteristic $2$. The proof in the odd characteristic case is given by making suitable modifications of
the arguments presented in \cite{walker2017total}, and the characteristic zero case is deduced from the prime characteristic case by using a variant of Hochster's metatheorem proved by Kurano-Roberts \cite{KRAdamsOps}.

\begin{thm}\label{thm:localTinyAnalog}
  The Total Rank Conjecture for Short Complexes, and hence the Total Rank Conjecture itself, holds for all commutative Noetherian algebras over a field.
\end{thm}

\begin{proof} By localizing at a prime of height $c$ in the support of  $P$, the general case reduces to the case of a tiny complex $P$ over a local ring $R$.

When $\chr(R) = 2$, the statement is Corollary \ref{MainCor}. Assume $\chr(R) \geq 3$. 
Then for each finite rank free $R$-module $E$, 
the short exact sequence 
$$
0 \to \L^2 E \to T^2 E \to S^2 E \to 0
$$
is in fact split exact, with the splitting of the left-hand map given by $e \otimes e' \mapsto \frac{1}{2} e \smsh e'$. 
It follows that
$$
0 \to \wt{\L^2}(P) \to \wt{T^2}(P) \to \wt{S^2}(P) \to 0
$$
is a split exact sequence of complexes and, in particular,
$$
h^\infty (P \otimes_R P) = h^\infty (\wt{S^2} P) + h^\infty (\wt{\L^2} P).
$$
By \cite[Theorem 3.1]{KRAdamsOps} we have
$$
\chi^\infty(\wt{S^2}P) - \chi^\infty(\wt{\L^2}P) = 2^d \cdot \chi^\infty(P).
$$
By Lemma \ref{lem:tensorIneq} we have $h^\infty (P \otimes_R P) \leq h^\infty (P) \beta(P)$ and
$h^\infty(P) = \chi^\infty(P)$ by Theorem \ref{thmRoberts}. 
Putting these facts together gives
$$
h^\infty(P) \cdot \beta(P) \geq 2^d \cdot h^\infty(P)
$$
and the result follows by dividing through by $h^\infty(P)$ (which is greater than $0$
by Theorem \ref{thmRoberts}).

Finally, assume that $R$ is an algebra over a field of characteristic $0$. By Lemmas $4.1$, $4.2$, and the discussion after the proof of Lemma $4.4$ in \cite{KRAdamsOps}, there exists a local ring $R'$ of prime characteristic $p$ and a complex $P'$ of free $R'$-modules satisfying:
    \begin{enumerate}
        \item $\dim R' = \dim R $,
        \item $\rank_{R'} (P'_i) = \rank_R (P_i) $ for all $i$, and
        \item $\ell_{R'} (H_i(P')) = \ell_R (H_i (P))  $ for all $i$.
    \end{enumerate}
    This means that $P'$ is a tiny complex over a field of positive characteristic, and hence 
    satisfies the Total Rank Conjecture. Since the total ranks of $P$ and $P'$ coincide and the dimensions of $R$ and $R'$ are equal, the complex $P$ also satisfies the Total Rank Conjecture.
\end{proof}

\section{Carlsson's Conjecture} \label{sec:Carlsson}

In this section, we apply the techniques used to prove the Total Rank Conjecture in characteristic two to investigate a conjecture of G.~Carlsson \cite[II.2]{carlsson1986free}. His
original conjecture concerned the $\Z/2$-cohomology of a finite CW complex that admits a free action by an elementary abelian $2$-group. The purely algebraic generalization of this conjecture
reads:

\begin{conj} [Carlsson's Conjecture]
Let $k$ be a field of characteristic $2$ and $R = k[x_1 , \dots , x_d] / (x_1^2 , \dots , x_d^2)$. For any bounded complex of finite rank free $R$-modules $P$ we have
$$
h(P) := \sum_i \dim_k H_i(P) \geq 2^d.
$$
\end{conj}

\begin{rem} \label{rem47}
    If $\chi(P) \neq 0$, then $h (P) \geq 2^d$ since 
    $$
    h(P) \geq |\chi (P)| = 2^d \left\vert \sum_{i} (-1)^i \rank_R (P_i) \right\vert.
    $$
We also note that Carlsson's conjecture is equivalent to the assertion that $\beta(P) \geq 2^d$ for $P \in \Perf^\fl(k[t_1, \dots ,t_d])$. This holds since, in characteristic $2$,
exterior algebras and symmetric algebras are Koszul-duals; see \cite[II.7]{carlsson1986free}. 
\end{rem}

\begin{prop}\label{prop:CarlssonsLength2} 
  Let $P = (\cdots \to 0 \to P_m \to \cdots \to P_0 \to 0)$ be a bounded complex of finite rank free modules over 
  $R = k[x_1 , \dots , x_d] / (x_1^2 , \dots , x_d^2)$ where $k$ is a field of characteristic $2$.
If $m = 1$, then $h(P) \geq 2^d$ and if $m = 2$, then $h (P) \geq 2^{d-1}$.
\end{prop}

\begin{proof} We may assume $P$ is minimal and that $\chi(P) = 0$ (see Remark \ref{rem47}).
  The arguments in Theorems \ref{thmB} and \ref{thm:012homology} give
  $$
h(P \otimes_R P) \geq  h(F(P)),  \text{ when $m = 1$ and } \,
h(P \otimes_R P) \geq  h_1(F(P)),  \text{ when $m = 2$.}
  $$
  Since $h(P \otimes_R P) \leq \beta(P) h(P)$ (see the proof of Lemma \ref{lem:tensorIneq}), it follows that
  $$
h(P) \geq  \frac{h(F(P))}{\beta(P)}, \text{ when $m = 1$, and } \,
h(P \otimes_R P) \geq  \frac{h_1(F(P))}{\beta(P)}, \text{ when $m = 2$.}
  $$
  Since $P$ is minimal and every element of the maximal ideal of $R$ squares to $0$, the differential in the complex $F(P)$ is zero, and thus
  $h_j(F(P)) = 2^d \beta_j(P)$ for all $j$.
  The first assertion is immediate and the second follows from $\beta_1(P) = \frac{\beta(P)}2$, which holds
 by the assumption $\chi(P) = 0$. 
\end{proof}

\begin{rem}
  The case $m = 1$ of this proposition was already known by work of Adem-Swan \cite[Corollary 2.1]{adem1995linear}. The proof presented here proves the (a priori) stronger inequality
  \begin{equation} \label{E47}
  h(P \otimes_R P) \geq 2^d \beta(P)
  \end{equation}
  in this case.
It thus seems reasonable to ask whether  \eqref{E47} holds for all $m$. Since $h(P \otimes_R P) \leq h(P) \beta(P)$ holds in general, this represents a stronger form of Carlsson's
Conjecture.  
\end{rem}

\bibliographystyle{amsalpha}
\bibliography{biblio}
\addcontentsline{toc}{section}{Bibliography}

\end{document}